\documentclass{amsart}

\usepackage{amssymb}
\usepackage[stretch=10,shrink=10]{microtype}
\usepackage[showframe=false, margin = 1.5 in]{geometry}
\usepackage{mathtools}
\usepackage[shortlabels]{enumitem}
\usepackage{tikz}
\usepackage{mathrsfs}
\usepackage{hyperref}
\usepackage{theoremref}

\renewcommand{\emptyset}{\varnothing}

\newcommand{\E}{\mathbf{E}}
\renewcommand{\P}{\mathbf{P}}
\newcommand{\Aa}{\mathcal{A}}

\newcommand{\TT}{\mathbb{T}}
\newcommand{\ZZ}{\mathbb{Z}}
\newcommand{\1}{\mathbf{1}}

\newcommand{\f}{\frac}

\newcommand{\RR}{\mathbb{R}}
\newcommand{\norm}[1]{\lVert #1 \rVert}

\makeatletter
\newcommand*\proc{{\mathpalette\bigcdot@{.7}}}
\newcommand*\bigcdot@[2]{\mathbin{\vcenter{\hbox{\scalebox{#2}{$\m@th#1\bullet$}}}}}
\makeatother
\DeclareMathOperator{\Poi}{Poi}
\DeclareMathOperator{\Ber}{Bernoulli}
\DeclareMathOperator{\Bernoulli}{Bernoulli}
\DeclareMathOperator{\Bin}{Bin}
\DeclarePairedDelimiter\abs{\lvert}{\rvert}%
\newcommand{\viu}[2]{a_{#1,#2}}
\newcommand{\vi}[1]{a_{#1}}
\usepackage{pict2e,picture}

\makeatletter
\newcommand{\pnrelbar}{%
  \linethickness{\dimen2}%
  \sbox\z@{$\m@th\prec$}%
  \dimen@=1.1\ht\z@
  \begin{picture}(\dimen@,.4ex)
  \roundcap
  \put(0,.2ex){\line(1,0){\dimen@}}
  \put(\dimexpr 0.5\dimen@-.2ex\relax,0){\line(1,1){.4ex}}
  \end{picture}%
}

\newcommand{\precneq}{\mathrel{\vcenter{\hbox{\text{\prec@neq}}}}}
\newcommand{\prec@neq}{%
  \dimen2=\f@size\dimexpr.04pt\relax
  \oalign{%
    \noalign{\kern\dimexpr.2ex-.5\dimen2\relax}
    $\m@th\prec$\cr
    \noalign{\kern-.5\dimen2}
    \hidewidth\pnrelbar\hidewidth\cr
  }%
}
\makeatother

\newtheorem{thm}{Theorem}
\newtheorem{lemma}[thm]{Lemma}
\newtheorem{prop}[thm]{Proposition}
\newtheorem{cor}[thm]{Corollary}

\newtheorem{question}[thm]{Open Question}

\theoremstyle{remark}
\newtheorem{remark}[thm]{Remark}
\theoremstyle{definition}
\newtheorem{define}[thm]{Definition}

\newcommand{\pgfprec}{\preceq_{\text{pgf}}}
\newcommand{\pgfsucc}{\succeq_{\text{pgf}}}
\newcommand{\stprec}{\preceq_{\text{st}}}

\newcommand{\icvprec}{\preceq_{\text{icv}}}

\newcommand{\id}{\mathrm{id}}

\begin{document}

\title{Stochastic orders and the frog model}

\author{Tobias Johnson}
\address{Department of Mathematics, New York University}
\email{tobias.johnson@nyu.edu}

\author{Matthew Junge}
\address{Department of Mathematics, Duke University}
\email{jungem@math.duke.edu}
\thanks{The first author was partially supported by NSF grant DMS-1401479.}

\keywords{Frog model, transience, recurrence, phase transition, stochastic orders, increasing concave order, probability generating function order}

\subjclass[2010]{60K35, 60J80, 60J10}

\date{April 6, 2017}

\begin{abstract}
The \emph{frog model} 
  starts with one active particle at the root of a graph and some number of dormant particles at all nonroot vertices.
Active particles follow independent random paths, waking all inactive particles they encounter. 
We prove that certain frog model statistics are monotone in the initial configuration for two nonstandard stochastic dominance relations: the increasing concave and the probability generating function orders.

This extends many canonical theorems.
We connect recurrence for random initial configurations to recurrence for deterministic configurations. 
Also,  the limiting shape of activated sites on the integer lattice respects both of these orders. 
Other implications include monotonicity results on transience of the frog model where the number of frogs per vertex decays away from the origin, on survival of the frog model with death, and on the time to visit a given vertex in any frog model. 
\vspace{12pt}

Le \emph{mod\`ele de la grenouille} se base sur un graphe contenant une particule active \`a sa racine et un certain nombre de particules dormantes sur ses autres sites. Les particules actives suivent des chemins al\'eatoires, et \'eveillent les particules inactives qu'elles rencontrent. Nous montrons que certaines statistiques des mod\`eles de grenouilles sont monotones en leur configuration initiale pour deux relations d'ordre: l'ordre concave croissant et l'ordre des fonctions g\'en\'eratrices.

Ces r\'esultats \'etendent de nombreux th\'eor\`emes canoniques. On \'etablit un lien entre la r\'ecurrence de la configuration initiale et la r\'ecurrence de configurations d\'eterministes. De plus, dans le graphe naturel des entiers, la forme limite des sites activ\'es respecte ces deux relations d'ordre. D'autres cons\'equences concernent des r\'esultats de monotonicit\'e sur la transience du mod\`ele de grenouille o\`u le nombre de grenouilles par site diminue lorsque l'on s'\'eloigne de l'origine, sur la survie dans un mod\`ele int\'egrant la mort, et sur les temps de visite a un site donn\'e dans n'importe quel mod\`ele de grenouille.
\end{abstract}

\maketitle

\section{Introduction} \label{sec:intro}

Let $G$ be a countable collection of vertices, one of which we distinguish as the root and call $\emptyset$. 
A general frog model $(\eta,S)$ starts with one active particle at $\emptyset$ and $\eta(v)$ dormant particles at each $v \neq \emptyset$. 
The $i$th particle at $v$ starting from its time of activation moves according to the path $S_\proc(v,i)$, with $S_0(v,i)$ assumed equal
to $v$. 
When an active particle visits a site containing dormant particles, \emph{all} of the dormant particles activate.
The particles move in discrete time, though this will be unimportant since most of the
properties of the frog model we consider depend only on
the particles' paths and not on the time they make their moves.
The particles are traditionally called frogs, and we continue  the zoomorphism.
Typically, $G$ is a graph, the frog paths $(S_\proc(v,i))_{v\in G,i\geq 1}$ are 
independent random walks, the frog counts $(\eta(v))_{v\in G}$ are either deterministic or i.i.d.,
and $(S_\proc(v,i))_{v\in G,i\geq 1}$ and $(\eta(v))_{v\in G}$ are independent of each other.
We will not belabor an example like the frog model with simple random walk paths on $\mathbb{Z}^d$
and i.i.d.-$\Poi(\mu)$ frogs per vertex by stating that the frog paths
are mutually independent, and that the frog counts and paths are independent.

Our main result is about two classes of frog model functionals we call \emph{icv} and \emph{pgf} \emph{statistics}.
The prime example is the number of visits to $\varnothing$ in the frog model $(\eta,S)$ over all
time, which we denote $r(\eta,S)$.
A realization of the frog model is called \emph{recurrent} if $r(\eta,S)= \infty$ 
and \emph{transient} otherwise.
In \cite{telcs1999}, the frog model with one sleeping frog per site and simple random walk paths
 is shown to be recurrent almost surely on $\mathbb Z^d$ for all $d$. This is further refined in \cite{random_recurrence}, which exhibits a threshold in $\alpha$ at which a frog model with $\Bernoulli(\alpha \|x\|^{-2})$ frogs at each $x \in \mathbb Z^d$ switches from transience to recurrence. 
 A similar phenomenon occurs 
when the walks have a bias in one direction: \cite{recurrence} finds that on $\mathbb Z$,
  the model is recurrent if and only if the number of sleeping frogs per site has infinite logarithmic moment.
  A sufficient condition for recurrence in this setting on $\mathbb{Z}^d$ was given in \cite{new_drift} and 
  improved on in \cite{01frog}.
  In our papers \cite{HJJ1,HJJ2,JJ3_log}, we prove that the frog model with simple random walk paths
  on the infinite $d$-ary tree $\TT_d$ 
  switches from transient to recurrent either when the density of frogs increases with
  $d$ held fixed, or when $d$ decreases with the density held fixed.

\subsection*{Statement of main theorem} 
Our main result is a comparison theorem relating icv and pgf statistics of a frog model
(see \thref{def:declining})
when we vary the distribution of the initial configuration $\eta$. Our original motivation was that
while the most convenient setting in our experience has Poisson-distributed frog counts, the most basic questions
assume a deterministic number of frogs per site.
As an example, in \cite{HJJ2} we showed the existence of a recurrence phase on the $d$-ary tree with 
Poisson frogs per site for any $d\geq 2$. 
This left open the existence of a recurrence phase
for initial conditions other than i.i.d.\ Poisson. 
For instance, for large enough $k$, is the frog model recurrent on the $d$-ary tree with
$k$ frogs per site?
With our previous tools, we could answer this question only for the case $d=2$ \cite{HJJ1},
but our comparison theorem tidily transfers the result from
Poisson to deterministic initial conditions (see \thref{cor:treek}).

If the initial condition $\eta(v)$ is dominated by $\eta'(v)$ in the usual stochastic order, 
then we can couple the corresponding frog models and deduce that $f(\eta, S)$ is dominated by $f(\eta',S)$ for any statistic $f$ that is increasing in $\eta$. 
This is not helpful for the problem described above, 
since we cannot relate a Poisson random variable to the constant~$k$ in this stochastic order.
Instead, our main theorem will show that if $\eta(v)$ is dominated by $\eta'(v)$ in a weaker
stochastic order, then $f(\eta,S)$ will be dominated by $f(\eta',S)$ for all $f$ in a certain
class of statistics.

First, we define the two weaker stochastic orders we use (see Section~\ref{sec:orders} for
further discussion).
For two random variables $X$ and $Y$ taking
values in $[0,\infty]$, we say that $X$ is dominated
by $Y$ in the \emph{increasing concave order} if $\E f(X)\leq \E f(Y)$ for all
bounded increasing concave functions $f\colon[0,\infty)\to \RR$, with $f(\infty)$
taken as $\lim_{x\to\infty} f(x)$. We denote this by $X\icvprec Y$. 
We say that $X$ is dominated by $Y$ in the \emph{probability generating
function order} if $\E t^X \geq \E t^Y$ for all $t\in(0,1)$, with $t^{\infty}$ interpreted
as $0$, and we denote this by $X\pgfprec Y$. From now on, we abbreviate increasing concave
order to \emph{icv order} and probability generating function order to \emph{pgf order}.
Since $x\mapsto 1-t^x$ is a bounded increasing concave
function, $X\icvprec Y$ implies $X\pgfprec Y$.
The icv order has come up several times in discrete probability, most notably
in first passage percolation \cite{BK,Marchand}. See also \cite{Zerner} for an application to random
walk in a random potential. In these papers, the relation
$\pi_1\icvprec \pi_2$ is referred to as $\pi_2$ being \emph{more variable} than $\pi_1$.
The only use of the pgf order that we know of in discrete probability is our own in \cite{HJJ1},
though see \cite{TRZ,LT} for some applications in signal processing and wireless networks
under the name \emph{Laplace transform order}.

Next, we define the classes of statistics covered by our main theorems.
Roughly speaking, we call a function of the frog model an \emph{icv statistic} if it increases when
a frog is added to the model, but when two frogs are added at the same vertex
it increases less than by the separate addition of each of them. The \emph{pgf statistics} form
a more restrictive class of frog model functionals that obey a higher order version of this property.
As we will prove in Section~\ref{sec:applications},
many natural frog model statistics fall into these classes, most notably
the count $r(\eta, S)$ of visits to the root.

Before we give the definitions, we will need to introduce some notation.
Let 
\begin{align*}
  \{\eta(v),S_\proc(v,i)\colon\,v\in G, i\geq 1\}
\end{align*}
be a deterministic collection of frog
counts and paths. For any path $P_\proc$, let $\sigma_{P_\proc}(\eta,S)$ denote a new frog model
with an extra frog of path $P_\proc$ added at $P_0$. That is, $\sigma_{P_\proc}(\eta,S)=(\eta',S')$, where
$\eta'$ is identical to $\eta$ except that $\eta'(P_0)=\eta(P_0)+1$, and $S'$ is identical
to $S$ except that $S'_\proc(P_0,\eta(P_0)+1)=P_\proc$. For any frog model statistic
$f(\eta,S)$, define
\begin{align*}
  \Delta_{P_\proc} f(\eta,S) = f(\sigma_{P_\proc}(\eta,S)) - f(\eta,S),
\end{align*}
the change in $f$ when a frog with path $P_\proc$ is added to the model.

\begin{define}\thlabel{def:declining}  
  Let $f$ be a functional of the frog model taking values in $[0,\infty]$.
  We call $f$ an \emph{icv statistic} if the following conditions hold for all $(\eta,S)$
  and all paths $P^1_\proc,\ldots,P^m_\proc$ starting at the same vertex:
  \begin{enumerate}[label=(\roman*)]
    \item  \label{main.condition}
      for $m= 1,2$,
      \begin{align*}%\label{eq:icv.condition}
        (-1)^m\Delta_{P^1_\proc}\cdots\Delta_{P^m_\proc}f(\eta,S)\leq 0,
      \end{align*}
      whenever all quantities in the expansion of the left hand side of the inequality are finite;
    \item if $f(\eta,S)=\infty$, then 
      $f\bigl(\sigma_{P^1_\proc}(\eta,S)\bigr)=\infty$; \label{infinite.increasing}
    \item if $f\bigl(\sigma_{P^1_\proc}\sigma_{P^2_\proc}(\eta,S)\bigr)=\infty$, then
      $f\bigl(\sigma_{P^i_\proc}(\eta,S)\bigr)=\infty$ for either $i=1$ or $i=2$.\label{infinite.concave}
  \end{enumerate}
  If $f$ satisfies conditions~\ref{infinite.increasing} and \ref{infinite.concave},
  and it satisfies condition~\ref{main.condition} not just for $m=1,2$ but for 
  all $m\geq 1$, then we call $f$ a \emph{pgf statistic}.
  In either case, we call the statistic \emph{continuous} if
  the condition
  \begin{align*}
    \eta_k(v)\nearrow\eta(v) \text{ as $k\to\infty$ for all $v\in G$}
  \end{align*}
  implies that
  $f(\eta_k,S)\nearrow f(\eta,S)$ as $k\to\infty$.
\end{define}
The $m=1$ case of \ref{main.condition} is 
the condition that $f$ increases when a new frog
is added. To make the $m=2$ case more transparent, we can restate it as
\begin{align*}
  f(\sigma_{P^1_\proc}\sigma_{P^2_\proc}(\eta,S)) - f(\sigma_{P^1_\proc}(\eta,S))
    -f(\sigma_{P^2_\proc}(\eta,S)) + f(\eta,S)\leq 0.
\end{align*}
Shifting terms around, we have the equivalent condition
\begin{align*}
  f(\sigma_{P^1_\proc}\sigma_{P^2_\proc}(\eta,S)) -f(\eta,S) 
    \leq \Bigl(f(\sigma_{P^1_\proc}(\eta,S)) -f(\eta,S)\Bigr)
    +\Bigl(f(\sigma_{P^2_\proc}(\eta,S)) - f(\eta,S)\Bigr),
\end{align*}
which states that the gain to the statistic by adding two frogs at the same
vertex is less than the combined gain of adding each frog separately, as we mentioned earlier.

Conditions~\ref{infinite.increasing} and~\ref{infinite.concave} are essentially the conditions
of being increasing and concave extended to apply when the statistic takes infinite values.
We also mention that condition~\ref{infinite.concave} implies the apparently stronger condition 
that for any $(\eta,S)$
and any paths $P^1_\proc,\ldots,P^m_\proc$ starting at the same vertex, if
$f\bigl(\sigma_{P^1_\proc}\cdots\sigma_{P^m_\proc}(\eta,S)\bigr)=\infty$, then
$f\bigl(\sigma_{P^i_\proc}(\eta,S)\bigr)=\infty$ for some $i$. Indeed,
suppose that $f\bigl(\sigma_{P^1_\proc}\cdots\sigma_{P^m_\proc}(\eta,S)\bigr)=\infty$.
Let $i_1,\ldots,i_\ell$ be any minimal set of indices such that
$f\bigl(\sigma_{P^{i_1}_\proc}\cdots\sigma_{P^{i_\ell}_\proc}(\eta,S)\bigr)=\infty$.
If $\ell\geq 2$, then with
$(\eta',S')=\sigma_{P^{i_3}_\proc}\cdots\sigma_{P^{i_\ell}_\proc}(\eta,S)$, we have
$f\bigl(\sigma_{P^{i_1}_\proc}\sigma_{P^{i_2}_\proc}(\eta',S')\bigr)=\infty$ even though 
$f\bigl(\sigma_{P^{i_1}_\proc}(\eta',S')\bigr)<\infty$ and $f\bigl(\sigma_{P^{i_2}_\proc}(\eta',S')\bigr)<\infty$.
This is impossible by condition~\ref{infinite.concave}. Hence
$\ell=1$, proving the stronger condition.

\begin{remark}
  The definition of icv and pgf statistics could be made in a more abstract setting.
  Suppose that $\mathcal{X}$ is the space of
  point process configurations on a space~$S$.
  For any $f\colon\mathcal{X}\to[0,\infty]$,
  $\chi\in\mathcal{X}$, and $x\in S$,  define $\Delta_xf(\chi)=f(\chi+\delta_x)-f(\chi)$.
  We could then call $f$ an icv statistic if for any $x_1,\ldots,x_n\in S$ and $n=1,2$,
  \begin{align*}
    (-1)^n\Delta_{x_1}\cdots\Delta_{x_n} f(\chi)\leq 0,
  \end{align*}
  and if statements analogous to condition~\ref{infinite.increasing}
  and \ref{infinite.concave} held. With $n=1,2$ replaced by $n\geq 1$,
  we would call $f$ a pgf statistic.
  Our previous definitions are equivalent to these when $S$ consists of paths in $G$.
  It is worth noting that being an icv statistic is the analogue in this discrete context
  of the usual notion of being increasing and concave, and being a pgf statistic is the analogue
  of having completely monotone derivative (see \eqref{eq:cmd}).
\end{remark}

With these stochastic orders and classes of statistics defined, we can finally
state our main result:

\begin{thm} \thlabel{thm:comparison}
 Assume that the frog paths $S_\proc(v,i)$ and counts $\eta(v)$ and $\eta'(v)$ are
 mutually independent for all $v$ and $i$, and that the paths $S_\proc(v,i)$ at a particular vertex~$v$ 
 are identically distributed for all $i$. 
 \begin{enumerate}[label = (\alph*)]
   \item If $f$ is a continuous icv statistic and
    $\eta(v)\icvprec\eta'(v)$ for all $v$, then $f(\eta,S)\icvprec f(\eta',S)$.
    \label{item:comparison.icv}
   \item If $f$ is a continuous pgf statistic and
    $\eta(v)\pgfprec\eta'(v)$ for all $v$, then $f(\eta,S)\pgfprec f(\eta',S)$.
    \label{item:comparison.pgf}
 \end{enumerate}
\end{thm}

The intuition behind the proof is that the extra frogs woken by the addition of two
frogs at some vertex is the union of the frogs woken by the addition of each frog separately.
This subadditivity property meshes neatly with concavity---for instance, the expected number
of visits to the root will increase concavely as frogs are added at a vertex---and somehow
this makes the frog model interact well with stochastic orders defined in terms of concave functions.

\subsection*{Applications}

As we mentioned, our main statistic of interest fits the criteria of \thref{thm:comparison}.
\begin{prop}\thlabel{lem:returns}
  The count $r(\eta,S)$ of visits to $\emptyset$ in the frog model $(\eta,S)$ is a continuous icv
  and pgf statistic of the frog model.
\end{prop}

This allows us to transfer many recurrence and transience results to different initial conditions.
%
%
%The following results are all consequences of \thref{thm:comparison}\ref{item:comparison.icv}. 
In the increasing concave order, the constant~$k$ dominates all mean~$k$ random variables.
\thref{thm:comparison}\ref{item:comparison.icv} and \thref{lem:returns} therefore imply the following:
\begin{cor} \thlabel{cor:det}
  Consider the frog model on a graph with mutually independent frog paths and i.i.d.\ frogs per site
  with common mean $\mu$. If this is almost surely recurrent, then for any integer 
  $k\geq\mu$, the same frog model with $k$ frogs per site is almost surely recurrent.
\end{cor}

This solves our problem of showing that the frog model on a $d$-ary tree with
deterministically $k$ frogs per site is recurrent for large enough $k$.
In more detail, \cite[Theorem~1]{HJJ2} establishes that on the $d$-ary tree
with i.i.d.-$\Poi(\mu)$ frogs per site, there is a critical value $\mu_c(d)$ such that the frog
model is recurrent~a.s.\ if $\mu>\mu_c(d)$ and transient~a.s.\ if $\mu<\mu_c(d)$. 
\thref{cor:det}, together with the estimates on $\mu_c(d)$ from \cite{JJ3_log}, give us the following result:
\begin{cor}\thlabel{cor:treek}
    For any $d\geq 2$, the frog model on $\mathbb T_d$ with $k$ frogs per site is almost surely
    recurrent for large enough $k$.
    For large enough $d$, the model is almost surely
    transient if $k<.24d$ and almost surely recurrent if $k>2.28d$.
\end{cor}

Another application of \thref{thm:comparison} 
concerns the transience regime of the $d$-ary tree. In \cite[Theorem~1]{HJJ1} we show 
that on $\mathbb T_d$ with one frog per site and simple random walk paths, the frog model is transient
for $d\geq 5$. 
An immediate corollary of \thref{thm:comparison} is transience for all other mean~1 configurations.

\begin{cor}\thlabel{cor:trans}
	For $d \geq 5$, the frog model on $\mathbb T_d$ with $\eta(v)$ frogs at each site, where $\E \eta(v) \leq 1$ for all $v \in \mathbb T_d$, is almost surely transient. 
\end{cor}
Our next application is to the frog model on $\mathbb{Z}^d$.
As mentioned earlier, \cite[Theorem~1.1]{random_recurrence} establishes the existence of a critical parameter
$0<\alpha_c(d)<\infty$ for the frog model with simple random walk paths on $\mathbb{Z}^d$
and initial configuration given by $\eta(x)\sim\Ber(p_x)$ such that
  \begin{enumerate}[label = (\roman*)]
    \item 
      if $p_x\leq\alpha/\norm{x}^2$ for $\alpha<\alpha_c(d)$ 
      and all sufficiently large $x$, then the model is transient with positive probability;
    \item
      if $p_x\geq\alpha/\norm{x}^2$ for $\alpha>\alpha_c(d)$ and all sufficiently large $x$, 
      then the model is transient with probability zero.
  \end{enumerate}
\thref{thm:comparison} allows us to extend part~(i) of this result to non-Bernoulli
distributions of sleeping frogs. Other results like \cite[Theorem~1.3]{random_recurrence} can be similarly
extended.

\begin{cor}\thlabel{cor:Zd.transience}
For all $\alpha < \alpha_c(d)$ and any $\bigl(\eta(x),\,x\in\mathbb Z^d\setminus\{0\}\bigr)$
satisfying $\E \eta(x)\leq \alpha/\norm{x}^2$ for sufficiently large $x$, 
the frog model on $\mathbb{Z}^d$ with simple random walk paths and initial 
configuration $\eta$ is transient with positive probability.
\end{cor}

%\HOX{some notational issues with this paragraph that we should discuss}
A fundamental result for the frog model on $\mathbb Z^d$ is that it has a limiting shape,
in the following sense. Let $\xi_n$ be
the set consisting of all lattice squares $x + (-1/2,1/2]^d$ such that $x\in\ZZ^d$ has 
been visited by time~$n$ in a frog model with i.i.d.-$\pi$ frogs per vertex.  
Theorem~1.1 from \cite{random_shape} establishes that for any dimension $d \geq 1$, there is a convex set $\mathbf{A} \subseteq \mathbb R^d$ depending on the distribution $\pi$
such that for any $0<\epsilon <1$,
$$\P\Bigl[(1-\epsilon)\mathbf{A}\subseteq \frac{ \xi_n}{n} \subseteq (1+\epsilon )\mathbf{A}
   \text{ for all sufficiently large $n$}\Bigr]=1.$$
Similar results were proven earlier for the discrete- and continuous-time model with one frog per site
in \cite{shape} and \cite{combustion}, respectively. 
We show that the limiting shape, $\mathbf{A}$, respects the icv and pgf orders. This mirrors the inequalities for the time constant for first passage percolation that are proven in \cite{BK}.

\begin{cor} \thlabel{cor:Zd.shape}
Let $\mathbf{A}$ and $\mathbf{A}'$ be the limiting shapes in the above sense for a frog model on $\mathbb Z^d$ with i.i.d.-$\pi$ and i.i.d.-$\pi'$ particles at each site, respectively.  If $\pi \pgfprec \pi'$, then $\mathbf{A} \subseteq \mathbf{A}'$.
\end{cor}

We also find applications to the frog model
with death, explored in \cite{phasetree,monotonic,po2}, where frogs have an independent chance $1-p$ of 
dying at each step. This is a frog model according to our general definition, taking the
frog paths to be stopped random walks. In this setting, the statistic of interest has been
the total number of sites visited, which undergoes a phase transition on the regular tree
from being finite a.s.\ to being
infinite with positive probability as $p$ grows. 
The model is said to \emph{die out} in the first case and to \emph{survive} in the second.
The number of sites visited is an icv and pgf statistic, as we show in \thref{lem:visited}, and we therefore obtain
the following result.

\begin{cor}\thlabel{cor:death}
  Let $\eta'(v) \pgfsucc \eta(v)$ be independent random variables indexed by the vertices $v$ of an arbitrary graph $G$. If the frog model with death on $G$ survives with $\eta(v)$ frogs at each $v$, then it survives with $\eta'(v)$ frogs at each $v$.
\end{cor}

All of the applications so far follow work with either of parts~\ref{item:comparison.icv} and 
\ref{item:comparison.pgf} of \thref{thm:comparison}, monotonicity in the icv and pgf orders, respectively. 
As \ref{item:comparison.pgf} is the more difficult to prove, one might wonder why we bother with it.
Our interest stems from the role of the pgf order in \cite{HJJ1} in proving recurrence
for the frog model on the binary tree with one frog per site. 
Our argument there works by showing that the number of visits
to the root is stochastically larger than any Poisson distribution in the pgf order. This hinges on 
%a monotonicity result 
\cite[Lemma~10]{HJJ1}, which shows that a certain operator is monotone with respect to the pgf order.
The proof there is an unsatisfying calculation that cannot easily be extended to a general $d$-ary tree.
But as we explain in \thref{rmk:pgf.applications}, this lemma and its analogues for $d\geq 3$
are now immediate corollaries of 
\thref{thm:comparison}\ref{item:comparison.pgf}. We hope this will be helpful in other problems
such as establishing recurrence for the frog model on a $3$-ary tree.

\subsection*{Questions}

We will give a few open problems on the theme of comparison theorems. 
A wider range of problems on the frog model are listed in \cite{HJJ1,HJJ2}.

We are interested in how sensitive the recurrence of the frog model is to the distribution of the
frog counts. We believe that recurrence depends not just on the mean number of frogs at each vertex,
but on the entire distribution.

\begin{question}
  Give an example where $r(\eta,S) = \infty$ a.s.\ and $r(\eta',S)< \infty$ a.s.\  
  with $\E \eta(v) = \E \eta'(v)$ for all $v$.
\end{question}
\noindent Specifically, we would like to know that with simple random walk paths on the binary tree and
i.i.d.-$\pi$ frogs per vertex with mean~$1$, 
the frog model is transient when $\pi$ is sufficiently unconcentrated.

Another question of ours is on a stronger version of \thref{cor:Zd.shape}. In \cite{BK}, van den Berg
and Kesten prove that in first passage percolation, strictly decreasing the passage time distribution
in the icv order yields a strictly smaller time constant (and hence a strictly smaller limiting shape).
Most of their work is in establishing the strictness.
\begin{question}
  Let $\mathbf{A}$ and $\mathbf{A}'$ be the limiting shapes for a frog model on $\ZZ^d$ with i.i.d.-$\pi$
  and i.i.d.-$\pi'$ initial sleeping frogs per site, respectively. Under what conditions
  does it hold that $\pi\precneq_{\text{icv}}\pi'$ implies $\mathbf{A}\subsetneq\mathbf{A}'$?
\end{question}
\noindent This cannot hold in full generality, because all 
choices of $\nu$ with sufficiently heavy tails have the same limiting shape, 
the $L^1$-ball in $\RR^d$ \cite[Theorem~1.2]{random_shape}. But it might hold under the assumption
that $\nu$ and $\nu'$ have finite expectations, for example. It might also
hold in full generality for the continuous-time frog model, but in this setting the shape theorem
has only been proven for one per site initial conditions.

Finally, we are interested in comparing frog models when the graph rather 
than the initial configuration is altered.
As a concrete question in this vein, we ask if the $d$-regular tree is the most transient graph
in the following sense:
\begin{question}\thlabel{q:most.transient}
  Suppose the frog model is transient on a $d$-regular graph~$G$ with simple random walks. 
  Is it necessarily transient on an infinite $d$-regular tree with simple random
  walk paths and the same initial conditions?
\end{question}

\subsection*{Acknowledgments}
We are grateful to Chris Hoffman for his general assistance and to Jonathan Hermon for a discussion
in 2014 that mentioned \thref{q:most.transient}. We thank Martin Zerner, who
pointed us to the previous uses of the icv order where $X\icvprec Y$ was referred to as
$X$ being more variable than $Y$. We thank Robin Pemantle, who pointed out to us
the interpretation of the pgf order in terms of thinnings mentioned on page~\pageref{page:thinnings}.
We are also grateful to Rapha\"el Lachi\`eze-Rey for helping us with the French-language abstract.

\section{Background material on stochastic orders}\label{sec:orders}

Let $\pi_1$ and $\pi_2$ be probability measures on the extended nonegative 
real numbers $[0,\infty]$, 
and let $X\sim\pi_1$ and $Y\sim\pi_2$.
The following three stochastic orders play a role in this paper:
\begin{description}
  \item[Standard stochastic order] $\pi_1\stprec\pi_2$ if $\E f(X)\leq \E f(Y)$ for all bounded increasing
    functions $f\colon[0,\infty)\to \RR$, with $f(\infty)$ taken as $\lim_{x\to\infty}f(x)$.
  \item[Increasing concave order] $\pi_1\icvprec\pi_2$ if $\E f(X)\leq \E f(Y)$ for all bounded increasing
    concave functions $f\colon[0,\infty)\to\RR$, with $f(\infty)$ taken as $\lim_{x\to\infty}f(x)$.
  \item[Probability generating function order] $\pi_1\pgfprec\pi_2$ if $\E t^X\geq \E t^Y$ for
    all $t\in(0,1)$, with $t^{\infty}$ interpreted as $0$.
\end{description}
We use $X\stprec Y$, $X\stprec\pi_2$, and $\pi_1\stprec Y$ all to mean that 
$\pi_1\stprec\pi_2$, and we do the same with the other
two orders.

We have listed these three stochastic orders in decreasing strength. That is,
\begin{align} \label{eq:order.strength}
  \pi_1\stprec\pi_2 \implies \pi_1\icvprec\pi_2\implies\pi_1\pgfprec\pi_2.
\end{align}
The first implication is obvious. For the second, the map $x\mapsto 1-t^x$ is 
a bounded increasing concave function for any $t\in(0,1)$, establishing that $\E t^X\geq\E t^Y$
for $t\in(0,1)$ if $X\icvprec Y$.

See \cite{SS} for a reference on stochastic dominance. 
We have made a few slight changes from the
usual definitions found there. First, in the standard and
icv orders, we have required our test functions to be bounded. This apparently weaker definition
is in fact equivalent to the usual one, as seen by approximating 
an unbounded increasing or increasing concave function
by a sequence of bounded ones. Second, we have restricted ourselves to probability measures supported
on nonnegative numbers, which is just a convenience. Last,
we have allowed our probability measures to take the value
$\infty$ with positive probability.
All of the standard results on stochastic orderings
are unaffected by this change. It is worth noting that
if $X\pgfprec Y$, then
$\P[X=\infty]\leq\P[Y=\infty]$. To see this, note that as $t\nearrow 1$, we have $t^x\to\1\{x<\infty\}$.
Thus, by the monotone convergence theorem,
\begin{align*}
  \E t^X\to \P[X<\infty]\quad\qquad \text{and}\qquad \quad
  \E t^Y\to \P[Y<\infty]
\end{align*}
as $t\nearrow 1$. Now $\E t^X\geq\E t^Y$ for $t\in(0,1)$ implies that
$\P[X<\infty]\geq\P[Y<\infty]$. By \eqref{eq:order.strength}, the conclusion also holds
under the assumption $X\stprec Y$ or $X\icvprec Y$.
We also mention that a similar argument with a limit as $t\searrow 0$ shows that if
$X\pgfprec Y$, then $\P[X=0]\geq\P[Y=0]$.

Roughly speaking, the standard order rewards distributions for being large,
while the icv order rewards them either for being large or for being concentrated.
The characterizations of these two orders in terms of couplings make this more precise:
$X\stprec Y$ if and only if $X$ and $Y$ can be coupled so that $X\leq Y$ a.s.\ \cite[Theorem~1.A.1]{SS}, 
and $X\icvprec Y$ if and only if $X$ and $Y$ can be coupled so that
$\E[X\mid Y]\leq Y$ a.s.\ \cite[Theorem~4.A.5]{SS}. Another useful equivalent condition
for $\pi_1\stprec\pi_2$ is that $\P[X>t]\leq\P[Y>t]$ for all $t$.

A function $\varphi$ is called \emph{completely monotone} if it is infinitely differentiable and
\begin{align}\label{eq:cmd}
  (-1)^n\varphi^{(n)}(x)\geq 0
\end{align}
for all $n\geq 0$ and $x$ in the domain of the function. 
By Bernstein's characterization of the completely monotone functions
as mixtures of functions of the form $e^{-ux}$, the statement
 $X\pgfprec Y$ holds
if and only if $\E \varphi(X)\geq\E\varphi(Y)$ for all completely monotone functions
$\varphi$ on $[0,\infty)$, or equivalently if $\E\varphi(X)\leq \E\varphi(Y)$
for $\varphi$ with completely monotone derivative \cite[Theorem~5.A.3]{SS}.
Unlike the other two orders, the pgf order does not have a characterization in terms of couplings
as far as we know, although it does have a probabilistic interpretation in terms of \emph{thinnings}.
\label{page:thinnings}
The \emph{$p$-thinning} of a nonnegative integer-valued random variable $N$ is a random
variable with the conditional distribution $\Bin(N,p)$ given $N$.
If $X$ and $Y$ are integer-valued, then $X\pgfprec Y$ if and only if the $p$-thinning of $X$ is more likely
than the $p$-thinning of $Y$ to be zero, for any $p\in[0,1]$. The advantage of the pgf order in our
experience is that one can test it by explicit calculations, as we did in our proof
of recurrence of the one-per-site frog model on the binary tree in \cite{HJJ1}.

We now give a pair of standard propositions, whose proofs we include in the appendix
for the sake of completeness. The first proposition 
lets us use test functions for icv and pgf dominance defined only on the nonnegative integers rather than
on all of $[0,\infty)$. For a function~$f$ on the integers, we define the difference operator
\begin{align}\label{eq:D.def}
  D f(k) &= f(k+1)-f(k).
\end{align}
We call this $D$ rather than the more common $\Delta$ to avoid ambiguity with a related
operator we will define in Section~\ref{sec:comparison}.
Repeated application of $D$ yields the following expression (see \cite[eq.~(1.97)]{EC1}):
\begin{align}\label{eq:D}
  D^n f(k) &= \sum_{i=0}^n(-1)^{n-i}\binom{n}{i}f(k+i).
\end{align}
\begin{prop}\thlabel{prop:discrete.ineq}
  Let $X$ and $Y$ take values on the extended nonnegative integers.
  In the following statements, we assume that $\varphi(k)$ is a bounded function
  on the nonnegative integers with a limit as $k\to\infty$, and we interpret
  $\varphi(\infty)$ as this limit.
  \begin{enumerate}[label = (\alph*)]
    \item It holds that $X\icvprec Y$ if and only if $\E\varphi(X)\leq\E\varphi(Y)$
      for functions $\varphi$ as above that satisfy $D \varphi(k)\geq 0$ and 
      $D^2\varphi(k)\leq 0$ for all $k\geq 0$.
      \label{item:discrete.ineq.icv}
    \item It holds that $X\pgfprec Y$ if and only if $\E\varphi(X)\leq\E\varphi(Y)$
      for functions $\varphi$ as above that satisfy $(-1)^nD^n \varphi(k)\leq 0$
      for all $n\geq 1$ and $k\geq 0$.
      \label{item:discrete.ineq.pgf}
  \end{enumerate}
\end{prop}

The next proposition shows that the maximal real- and integer-valued distributions
in the icv order with a given expectation are the distributions that are as concentrated as possible.

\begin{prop}\thlabel{prop:maximals}\ 
  \begin{enumerate}[label = (\alph*)]
    \item A nonnegative random variable satisfies $X\icvprec c$ for any $c\geq\E X$.\label{part:Jensen}
    \item Suppose $X$ takes nonnegative integer values and $\E X\in[k,k+1]$ for an integer~$k$.
      Let $Y$ be a random variable taking values in $\{k,k+1\}$ and satisfying $\E X\leq \E Y$.
      Then $X\icvprec Y$.\label{part:Bernoulli}
  \end{enumerate}
\end{prop}

\section{Proof of the comparison theorem}\label{sec:comparison}

We start by giving some technical facts about icv and pgf statistics and about
the operator $\Delta_{P_\proc}$.
As with the related operator~$D$ defined in \eqref{eq:D.def},
this operator can be applied repeatedly and expanded in the following way, which resembles
\eqref{eq:D}.
Let $P^1_\proc,\ldots,P^n_\proc$ be frog paths, and let $U=\{u_1,\ldots,u_j\}\subseteq[n]$, where we use
the notation $[n]=\{1,\ldots,n\}$.
Define
\begin{align}\label{eq:sigma.U}
  \sigma_U(\eta,S) = \sigma_{P^{u_1}_\proc}\cdots\sigma_{P^{u_j}_\proc}(\eta,S),
\end{align}
the frog model $(\eta,S)$ with the addition of frogs $P^{u_1}_\proc,\ldots,P^{u_j}_\proc$.
If $U$ is empty, take $\sigma_U(\eta,S)=(\eta,S)$.
Using this notation,
\begin{align}
  \Delta_{P^1_\proc}\cdots\Delta_{P^n_\proc}f(\eta, S)
    &= \sum_{U\subseteq[n]} (-1)^{n-\abs{U}}f(\sigma_U(\eta,S)).
  \label{eq:Delta.pp}
\end{align}
This can be proven by the same argument used in \cite[eq.~(1.97)]{EC1}.

Our next proposition states that icv statistics are closed under composition
with increasing concave functions, and that pgf statistics are closed under composition
with functions with completely monotone derivatives.
Though we could not find these statements
in existing literature, they are essentially equivalent to the well known fact that
increasing concave functions and functions with completely monotone derivatives both form
closed families
under composition. We give a full proof in the appendix.

\begin{prop}\thlabel{lem:composition}
  Let $\varphi\colon[0,\infty)\to[0,\infty)$ be bounded, and interpret $\varphi(\infty)$
  as $\lim_{x\to\infty}\varphi(x)$.
  \begin{enumerate}[label = (\alph*)]
  \item Suppose that $\varphi$ is increasing and concave. If $f$ is an icv statistic,
  then $\varphi\circ f$ is also an icv statistic. \label{item:composition.icv}
  \item Suppose that $\varphi$ has completely monotone derivative; that is, it satisfies
  \begin{align*}
    (-1)^n\varphi^{(n)}(x)\leq 0
  \end{align*}
  for all $n\geq 1$. If $f$ is a pgf statistic, then $\varphi\circ f$ is also
  a pgf statistic.\label{item:composition.pgf}
  \end{enumerate}
\end{prop}

In the next lemma, we show that icv and pgf statistics are monotone in the distribution of frogs
at a \emph{single} vertex.
\begin{lemma}\thlabel{lem:change.one}
  Make the assumptions of \thref{thm:comparison} on the distribution of frog paths $S_\proc(v,i)$
  and counts $\eta(v)$ and $\eta'(v)$. Also assume that $\eta$ and $\eta'$ have identical
  distributions at all but one vertex~$v_0$.
  \begin{enumerate}[label = (\alph*)]
    \item If $f$ is an icv statistic and $\eta(v_0)\icvprec\eta'(v_0)$, then
      $f(\eta,S)\icvprec f(\eta',S)$.\label{item:change.icv}
    \item If $f$ is a pgf statistic and $\eta(v_0)\pgfprec\eta'(v_0)$, then
      $f(\eta,S)\pgfprec f(\eta',S)$.\label{item:change.pgf}
  \end{enumerate}
\end{lemma}
\begin{proof}
  Define $\eta_k$ to be the same as $\eta$ except that $\eta_k(v_0)=k$.
  Let $W(k)=f(\eta_k,S)$. By our assumptions, $\eta(v_0)$ and $\eta'(v_0)$ are independent of $W(k)$,
  and hence 
  \begin{align}\label{eq:W.dist}
    W(\eta(v_0))\sim f(\eta,S)\qquad\text{and}\qquad W(\eta'(v_0))\sim f(\eta',S).
  \end{align}
  
  We start with the proof of \ref{item:change.icv}.
  Let $\varphi\colon [0,\infty)\to[0,\infty)$ be an arbitrary bounded increasing concave function,
  and let $h(k)=\E\varphi(W(k))$ for $k\in\{0,1,\ldots\}$.
  As
  \begin{align*}
    \E h(\eta(v_0)) = \E \varphi(f(\eta,S))
    \qquad\text{and}\qquad
    \E h(\eta'(v_0)) = \E \varphi(f(\eta',S))
  \end{align*}
  by \eqref{eq:W.dist},
  our goal is to show that $\E h(\eta(v_0))\leq \E h(\eta'(v_0))$.
  If we can show that $D h(k)\geq 0$ and $D^2 h(k)\leq 0$, then
  this follows immediately from \thref{prop:discrete.ineq}\ref{item:discrete.ineq.icv}
  and the assumption that $\eta(v_0)\icvprec\eta'(v_0)$.
  
  By definition of $h$ and $W$,
  \begin{align*}
    D h(k) = \E\bigl[ \varphi(f(\eta_{k+1},S)) - \varphi(f(\eta_k,S)) \bigr].
  \end{align*}
  As $f(\eta_{k+1},S)\geq f(\eta_k,S)$ and $\varphi$ is increasing, one can see directly
  that $D h(k)\geq 0$, but it is more instructive to derive this as a consequence of
  \thref{lem:composition}.
  Let $P_\proc$ be an independent copy of $S_\proc(v_0,1)$. By our assumption that
  $(S_\proc(v_0,i))_{i\geq 1}$ are i.i.d.\ and independent of the other frog paths, 
  the frog model $(\eta_{k+1},S)$ is distributed the same
  as $\sigma_{P_\proc}(\eta_k,S)$. Hence
  \begin{align}
    \begin{split}
    D h(k) &= \E\bigl[ \varphi(f(\sigma_{P_\proc}(\eta_k,S)) - \varphi(f(\eta_k,S)) \bigr]\\
      &= \E\bigl[ \Delta_{P_\proc}(\varphi\circ f)(\eta_k,S) \bigr],
    \end{split}
    \label{eq:Delta1}
  \end{align}
  which is nonnegative since $\varphi\circ f$ is an icv statistic by 
  \thref{lem:composition}\ref{item:composition.icv}.
  Similarly, if $P^1_\proc$ and $P^2_\proc$ are independent copies of $S_\proc(v_0,i)$,
  \begin{align}
    \begin{split}
    D^2 h(k) &= \E\bigl[ \varphi(f(\eta_{k+2},S)) - 2\varphi(f(\eta_{k+1},S)) + \varphi(f(\eta_k,S)) 
                          \bigr]\\
                  &= \E\bigl[ \varphi(f(\sigma_{P^1_\proc}\sigma_{P^2_\proc}(\eta_k,S)))
                    - \varphi(f(\sigma_{P^1_\proc}(\eta_k,S))) - \varphi(f(\sigma_{P^2_\proc}(\eta_k,S)))
                    + \varphi(f(\eta_k,S)) \bigr]\\
                  &= \E\bigl[ \Delta_{P^1_\proc}\Delta_{P^2_\proc}(\varphi\circ f)(\eta_k, S) \bigr]
                     \leq 0.
    \end{split}
    \label{eq:Delta2}
  \end{align}
  This concludes the proof of part~\ref{item:change.icv}.
  
  The proof of \ref{item:change.pgf} is essentially the same.
  We take $\varphi(x)=1-t^x$ for arbitrary $t\in(0,1)$ and define
  $h(k)=\E\varphi(W(k))$ as before. This time, we need to show that
  $(-1)^nD^nh(k)\leq 0$ for all $n\geq 1$. By \thref{prop:discrete.ineq}\ref{item:discrete.ineq.pgf}
  and the assumption that $\eta(v_0)\pgfprec\eta'(v_0)$, it follows from this that
  $\E h(\eta(v_0))\leq \E h(\eta'(v_0))$, and hence that 
  $\E \varphi(f(\eta,S))\leq \E \varphi(f(\eta',S))$.
  
  Let $P^1_\proc,\ldots,P^n_\proc$ be independent copies of $S_\proc(v_0,i)$.
  Using the notation of \eqref{eq:sigma.U}, for any $U\subseteq[n]$ with $\abs{U}=i$,
  the frog model $(\eta_{k+i},S)$ is distributed identically to $\sigma_U(\eta_k,S)$.
  We now generalize \eqref{eq:Delta1} and \eqref{eq:Delta2} by applying
  \eqref{eq:D} and \eqref{eq:Delta.pp} to get
  \begin{align*}
    D^n h(k) &= \E\Biggl[ \sum_{i=0}^n(-1)^{n-i}\binom{n}{i} \varphi(f(\eta_{k+i},S)) \Biggr]\\
      &=\E \Biggl[ \sum_{i=0}^n (-1)^{n-i}\sum_{\substack{U\subseteq[n]\\\abs{U}=i}}
                                  \varphi(f(\sigma_U(\eta_k,S)))\Biggr]\\    
    &=\E\bigl[ \Delta_{P^1_\proc}\cdots\Delta_{P^n_\proc}(\varphi\circ f) \bigr].
  \end{align*}
  As $\varphi\circ f$ is a pgf statistic by 
  \thref{lem:composition}\ref{item:composition.pgf}, this shows that
  $(-1)^nD^n h(k)\leq 0$, completing the proof.
\end{proof}

\begin{proof}[Proof of Theorem~\ref{thm:comparison}]
  The basic idea is that \thref{lem:change.one} proves the result when
  $\eta$ and $\eta'$ have the same distribution
  at all but finitely many vertices, with the general case following from a limit
  argument relying on the continuity assumption.
  Recall from \thref{def:declining} that we call a frog model statistic continuous
  if the upward convergence of frog counts implies the upward convergence of the statistic.
  
  Let $G_1\subseteq G_2\subseteq\cdots$ be finite sets of vertices whose union is $G$.
  We use $\eta|_{G_k}$ and $\eta'|_{G_k}$ to denote restrictions to $G_k$. That is,
  $\eta|_{G_k}(v)=\eta(v)\1\{v\in G_k\}$.
  Since $\eta|_{G_k}$ and $\eta'|_{G_k}$ differ at only finitely many vertices, 
  the repeated application of \thref{lem:change.one}
  proves that in case~\ref{item:comparison.icv},
  \begin{align}\label{eq:finite.case.icv}
    f(\eta|_{G_k},S)\icvprec f(\eta'|_{G_k},S),
  \end{align}
  and in case~\ref{item:comparison.pgf},
  \begin{align}\label{eq:finite.case.pgf}
    f(\eta|_{G_k},S)\pgfprec f(\eta'|_{G_k},S).
  \end{align}
    
  Now, we let $\varphi$ be a test function and try to show that 
  \begin{align}\label{eq:test.function}
    \E\varphi\bigl(f(\eta,S)\bigr)\leq\E\varphi\bigl(f(\eta',S)\bigr).
  \end{align}  
  For case~\ref{item:comparison.icv}, let $\varphi\colon[0,\infty)\to[0,\infty)$ be a bounded 
  increasing concave function. It suffices to show \eqref{eq:test.function} for such functions $f$,
  as any arbitrary bounded 
  increasing concave function can be shifted to take nonnegative values. 
  In case~\ref{item:comparison.pgf},
  let $\varphi(x)= 1-t^x$ for some $t\in(0,1)$.
  Interpreting $\varphi(\infty)$ as $\lim_{x\to\infty}\varphi(x)$ as usual,
  it holds by the continuity assumption that
  \begin{align*}
    \varphi\bigl(f(\eta|_{G_k},S)\bigr)\nearrow \varphi\bigl(f(\eta,S)\bigr)\text{ a.s.\qquad and \qquad}
    \varphi\bigl(f(\eta'|_{G_k},S)\bigr)\nearrow\varphi\bigl(f(\eta',S)\bigr)\text{ a.s.}
  \end{align*}
  as $k\to\infty$.
  By the monotone convergence theorem,
  \begin{align*}
    \E\varphi\bigl(f(\eta|_{G_k},S)\bigr)\to \E\varphi\bigl(f(\eta,S)\bigr)\qquad\quad \text{and}\qquad \quad
       \E\varphi\bigl(f(\eta'|_{G_k},S)\bigr)\to \E\varphi\bigl(f(\eta',S)\bigr)
  \end{align*}
  as $k\to\infty$.
  By \eqref{eq:finite.case.icv} or \eqref{eq:finite.case.pgf}, we have
  $\E\varphi\bigl(f(\eta|_{G_k},S)\bigr)\leq \E\varphi\bigl(f(\eta'|_{G_k},S)\bigr)$, and
  this proves \eqref{eq:test.function}.
\end{proof}

\section{Applications of the comparison theorem}\label{sec:applications}

To apply \thref{thm:comparison}, we first need to find some icv and pgf statistics.
The following two lemmas highlight particular circumstances where we can draw conclusions
about the difference operators applied to a statistic.
For any frog model $(\eta,S)$, we define $\kappa_v(\eta, S)$ as the frog model 
resulting from deleting all frogs that start at $v$ from $(\eta, S)$. Formally,
$\kappa_v(\eta,S)=(\eta',S)$, where $\eta'$ is identical to $\eta$
except that $\eta'(v)=0$.

\begin{lemma}\thlabel{lem:stat.builder}
  Let $f$ be a frog model statistic taking values only in $\{0,1\}$.
  Suppose that for some vertex~$v$ and all $(\eta,S)$,
  \begin{align}\label{eq:subadditivity}
    f(\eta,S) = \max\Bigl\{f\bigl(\kappa_{v}(\eta,S)\bigr),\; f\bigl(\sigma_{S_\proc^1}(\kappa_v(\eta,S))\bigr),
    \ldots, f\bigl(\sigma_{S_\proc^{\eta(v)}}(\kappa_v(\eta,S))\bigr) \Bigr\}.
  \end{align}
  Then for any paths $P^1_\proc,\ldots,P^n_\proc$ originating at $v$ and all $(\eta,S)$,
  \begin{align*}
    (-1)^n\Delta_{P^1_\proc}\cdots\Delta_{P^n_\proc}f(\eta,S) \leq 0.
  \end{align*}
\end{lemma}
\begin{proof}
  Fix any $(\eta,S)$ and paths $P^1_\proc,\ldots,P^n_\proc$ starting at $v$.
  Since $f$ is increasing as additional frogs are added by \eqref{eq:subadditivity},
  if $f(\eta,S)=1$, then $f(\sigma_U(\eta,S))=1$ for any $U\subseteq[n]$,
  using the notation given in \eqref{eq:sigma.U}. Hence
  \begin{align*}
    \Delta_{P^1_\proc}\cdots\Delta_{P^n_\proc}f(\eta,S)=0,
  \end{align*}
  and so the lemma holds in this case.
  If $f(\eta,S)=0$, define $b_i=f(\sigma_{P_\proc^i}(\eta,S))$, the statistic after adding
  $P_\proc^i$ to $(\eta,S)$. By \eqref{eq:subadditivity}, $f(\sigma_U(\eta,S))=1$ if
  and only if $b_i=1$ for some $i\in U$. Thus
  \begin{align*}
    f(\sigma_U(\eta,S)) &= \max_{i\in U}b_i = 1 - \prod_{i\in U}(1-b_i),
  \end{align*}
  and \eqref{eq:Delta.pp} gives
  \begin{align*}
    \Delta_{P^1_\proc}\cdots\Delta_{P^n_\proc} f(\eta,S)
      &= \sum_{U\subseteq[n]}(-1)^{n-\abs{U}}\biggl(1 - \prod_{i\in U}(1-b_i)\biggr)\\
      &= \sum_{U\subseteq[n]}(-1)^{n-\abs{U}} - \sum_{U\subseteq[n]}(-1)^{n-\abs{U}}\prod_{i\in U}(1-b_i).
  \end{align*}
  The first sum is the expansion of $(1-1)^n$ and hence is zero. For the second sum,
  \begin{align*}
    \sum_{U\subseteq[n]}(-1)^{n-\abs{U}}\prod_{i\in U}(1-b_i) &= \prod_{i=1}^n\bigl((1-b_i) - 1\bigr)\\
      &= \prod_{i=1}^n(-b_i) = (-1)^n\1\{b_1=\cdots=b_n=1\}.
  \end{align*}
  Thus
  \begin{align*}
    \Delta_{P^1_\proc}\cdots\Delta_{P^n_\proc}f(\eta,S) &= (-1)^{n+1}\1\{b_1=\cdots=b_n=1\},
  \end{align*}
  yielding $(-1)^n\Delta_{P^1_\proc}\cdots\Delta_{P^n_\proc}f(\eta,S)\leq 0.$
\end{proof}

\begin{lemma}\thlabel{lem:stat.builder.sum}
  Let $f$ be a frog model statistic taking values in $[0,\infty]$.
  Suppose that for some vertex~$v$ and all $(\eta,S)$,
  \begin{align}\label{eq:additivity}
    f(\eta,S) = \sum_{i=1}^{\eta(v)}f\bigl(\sigma_{S_\proc(v,i)}(\kappa_{v}(\eta,S))\bigr).
  \end{align}
  Then conditions \ref{main.condition}, \ref{infinite.increasing}, and
  \ref{infinite.concave} of being a pgf statistic hold for all $(\eta,S)$ and
  paths $P^1_\proc,\ldots,P^n_\proc$ originating at $v$.
\end{lemma}
\begin{proof}
  Fix $(\eta,S)$ and paths $P^1_\proc,\ldots,P^m_\proc$ originating at $v$.
  Let $b_i=f\bigl(\sigma_{P^i_\proc}(\kappa_{v}(\eta,S))\bigr)$, and let $b = f(\eta,S)$.
  As in the previous lemma, we use the notation $\sigma_U$ 
  given in \eqref{eq:sigma.U}.
  From \eqref{eq:additivity}, for any $U\subseteq[m]$,
  \begin{align*}
    f(\sigma_U(\eta,S)) = b + \sum_{i\in U} b_i.
  \end{align*}
  First, consider the case that $b,b_1,\ldots,b_m<\infty$.
  Let 
  \begin{align*}%\label{eq:D.generalized}
    D_a h(x) = h(x+a)-h(x),
  \end{align*}
  generalizing the difference operator $D$ given in \eqref{eq:D.def}.
  The operator $D_a$ satisfies
  \begin{align*}
    D_{a_1}\cdots D_{a_m} h(x) &= \sum_{U\subseteq [m]}(-1)^{m-\abs{U}} h\Biggl(x+\sum_{i\in U}a_i\Biggr),
  \end{align*}
  proven identically as \eqref{eq:D} and \eqref{eq:Delta.pp}. Comparing with \eqref{eq:Delta.pp},
  we have
  \begin{align*}
    \Delta_{P^1_\proc}\cdots\Delta_{P^m_\proc}f(\eta,S) = D_{b_1}\cdots D_{b_m}\id(b),
  \end{align*}
  where $\id(x)=x$. For $m\geq 2$, this is equal to zero, as follows from the second and higher
  partial derivatives of $\id$ being zero. For $m=1$, this is nonnegative, because $b_1\geq 0$
  by our assumption that $f$ takes nonnegative values. This shows that condition~\ref{main.condition}
  holds. Conditions~\ref{infinite.increasing} and~\ref{infinite.concave} follow immediately from 
  \eqref{eq:additivity}.  
\end{proof}
Next, we show that icv and pgf statistics are closed under summation.
\begin{lemma}\thlabel{lem:linearity}
  Let $f_1,f_2,\ldots$ be frog model statistics taking values in $[0,\infty]$,
  and let $f=\sum_{i=1}^{\infty} f_i$. If $f_1,f_2,\ldots$ are icv statistics, then
  $f$ is an icv statistic, and if $f_1,f_2,\ldots$ are pgf statistics, then $f$ is
  a pgf statistic.  
\end{lemma}
\begin{proof}
  Suppose that $f_1,f_2,\ldots$ are pgf statistics.
  We have
  \begin{align*}
    (-1)^m\Delta_{P^1_\proc}\cdots\Delta_{P^m_\proc}f(\eta, S)
      &= \sum_{i=1}^{\infty}(-1)^m\Delta_{P^1_\proc}\cdots\Delta_{P^m_\proc}f_i(\eta, S)
  \end{align*}
  when all quantities in the expansion of the left hand side are finite.
  All quantities on the right hand side are finite as well, and all are
  nonpositive since $f_i$ is a pgf statistic for all $i$, which
  shows that condition~\ref{main.condition} is satisfied. For condition~\ref{infinite.increasing},
  we note that by conditions~\ref{main.condition} and~\ref{infinite.increasing} applied to $f_i$,
  \begin{align*}
    \sum_{i=1}^{\infty} f_i(\eta,S) \leq \sum_{i=1}^{\infty}f_i\bigl(\sigma_{P^1_\proc}(\eta,S)\bigr)
  \end{align*}
  even when the left hand side is infinite. And for condition~\ref{infinite.concave},
  by conditions~\ref{main.condition} and~\ref{infinite.concave} applied to $f_i$,
  \begin{align*}
    \sum_{i=1}^{\infty} f_i\bigl(\sigma_{P^1_\proc}\sigma_{P^2_\proc}(\eta,S)\bigr)
      &\leq \sum_{i=1}^{\infty}f_i\bigl(\sigma_{P^1_\proc}(\eta,S)\bigr) +
         \sum_{i=1}^{\infty}f_i\bigl(\sigma_{P^2_\proc}(\eta,S)\bigr),
  \end{align*}
  even when the left hand side is infinite.  The proof is identical for the icv case.
\end{proof}

Now, we apply \thref{lem:stat.builder,lem:stat.builder.sum} to prove that various frog model
statistics are pgf (and hence also icv). We start with some counts of visited sites.
\begin{prop}\thlabel{lem:visited}
  For $t\in\mathbb{N}\cup\{\infty\}$ and any nonroot vertex~$u$,
  let $\viu{t}{u}(\eta,S)$ be an indicator on site~$u$ being visited in the frog model $(\eta,S)$
  by time~$t$. Let $\vi{t}(\eta,S)$ be the total number of nonroot sites visited
  by time~$t$. Then both $\viu{t}{u}$ and $\vi{t}$ are continuous icv and pgf statistics.
\end{prop}
\begin{proof}
  First, we show that $\viu{t}{u}$ is a pgf statistic.  
  Let $v\neq u$ be a nonroot vertex.
  We claim that \eqref{eq:subadditivity} is satisfied for $\viu{t}{u}$ and vertex~$v$.
  To prove this, we need to show that $\viu{t}{u}(\eta,S)=1$ if and only if
  $\viu{t}{u}\bigl(\kappa_v(\eta,S)\bigr)=1$ or
  $\viu{t}{u}\bigl(\sigma_{S_\proc(v,i)}(\kappa_{v}(\eta,S))\bigr)=1$ for some $1\leq i\leq \eta(v)$. 
  First, suppose
  $\viu{t}{u}(\eta,S)=1$. This means that there exists a sequence of frogs
  starting with the initial frog and ending with a frog that visits $u$ such that
  each frog activates the next one in the sequence and the
  combined path length is at most $t$. We can assume that all frogs in this sequence
  originate at different vertices, since otherwise we could cut out portions of the sequence
  to make this true.
  If the sequence includes the
  $i$th frog at $v$, then $\viu{t}{u}\bigl(\sigma_{S_\proc(v,i)}(\kappa_{v}(\eta,S))\bigr)=1$. 
  If the sequence does not include any of the frogs
  originating at $v$, then $\viu{t}{u}(\kappa_{v}(\eta,S))=1$.
  The converse is obvious, since if $u$ is visited in time~$t$ by
  $\sigma_{S_\proc(v,i)}(\kappa_{v}(\eta,S))$ or $\kappa_v(\eta,S)$, then it is also
  visited by $(\eta,S)$.
  
  Thus, \thref{lem:stat.builder} applies and shows that
  \begin{align*}
    (-1)^n\Delta_{P^1_\proc}\cdots\Delta_{P^n_\proc}\viu{t}{u}(\eta,S)\leq 0
  \end{align*}
  for any $(\eta,S)$, nonroot vertices $u\neq v$, and paths $P^1_\proc,\ldots,P^n_\proc$ originating
  at $v$. In the case $u=v$,
  \begin{align*}
    (-1)^n\Delta_{P^1_\proc}\cdots\Delta_{P^n_\proc}\viu{t}{u}(\eta,S) = 0,
  \end{align*}
  as the addition of extra frogs at $u$ does not affect whether $u$ is visited. Thus, $\viu{t}{u}$
  is a pgf statistic for any nonroot vertex $u$.
  It follows by \thref{lem:linearity} that $\vi{t}$ is a pgf statistic, since we can express
  it as
  \begin{align}\label{eq:a^t.sum}
    \vi{t}(\eta,S) &= \sum_{v\neq \emptyset} \viu{t}{v}(\eta,S).
  \end{align}
  
  It remains to prove that $\viu{t}{u}$ and $\vi{t}$ are continuous.
  This holds because any frog 
  woken in $(\eta,S)$ relies only on a finite sequence of frogs to wake it. More formally,
  suppose that the components of $\eta_k$ converge upwards to $\eta$ as $k\to\infty$.
  If $\viu{t}{u}(\eta,S)=1$, then for large enough $k$ we have $\viu{t}{v}(\eta_k,S)=1$, because the sequence of
  frogs visiting $u$ in time~$t$ is finite. Thus $\viu{t}{u}(\eta_k,S)\nearrow \viu{t}{u}(\eta,S)$
  as $k\to\infty$, meaning that $\viu{t}{u}$ is continuous. 
  By \eqref{eq:a^t.sum} and monotone convergence, $\vi{t}(\eta_k,S)\nearrow \vi{t}(\eta,S)$, and $\vi{t}$
  is continuous as well.
\end{proof}

\begin{proof}[Proof of \thref{lem:returns}]
  This is a slightly more complicated version of the previous proof.
  For any nonroot vertex~$u$, let $r_u(\eta,S)$ be the number of visits to the root
  in the frog model $(\eta,S)$ by frogs originating at vertex~$u$. Fix some vertex~$v\neq u$
  and frog paths $P^1_\proc,\ldots,P^m_\proc$ originating at $v$, and fix $(\eta,S)$.
  Let $N$ be the total number of visits to the root by paths $S_\proc(u,1),\ldots,S_\proc(u,\eta(u))$.
  Since adding extra frogs at vertex~$v$ affects only whether $u$ is activated, not the number
  of frogs returning from it if activated,
  \begin{align}\label{eq:diff.vert}
    r_u(\sigma_U(\eta,S)) = N \viu{\infty}{u}(\sigma_U(\eta,S))
  \end{align}
  for any $U\subseteq [m]$.
  It then follows from \thref{lem:visited} that 
  $(-1)^m\Delta_{P^1_\proc}\cdots\Delta_{P^m_\proc}r_u(\eta,S)\leq 0$ when all terms in this expansion
  are finite (that is, when $N<\infty$). From the $N=\infty$ case of
  \eqref{eq:diff.vert}, it also follows that
  conditions~\ref{infinite.increasing} and \ref{infinite.concave} hold for all $(\eta,S)$
  and all paths $P^1_\proc,\ldots P^m_\proc$ originating at $v$.
  
  Next, we consider the case $v=u$. 
  If $u$ is visited by $(\eta,S)$, then the number of visits to the root originating at $u$
  is the number of visits to the root by paths $S_\proc(u,1),\ldots,S_\proc(u,\eta(u))$.
  Since modifying the frogs at $u$ does not change whether $u$ is visited, this implies that
  if $\viu{\infty}{u}(\eta,S)=1$,
  \begin{align*}
    r_u(\eta,S) &= \sum_{i=1}^{\eta(u)} r_u\bigl(\sigma_{S_\proc(u,i)}(\kappa_{v}(\eta,S))\bigr).
  \end{align*}
  This equation also holds if $\viu{\infty}{u}(\eta,S)=0$, since then both sides are zero.
  Hence the conditions of \thref{lem:stat.builder.sum} are satisfied,
  and conditions~\ref{main.condition}, \ref{infinite.increasing}, and \ref{infinite.concave}
  are satisfied for all $(\eta,S)$ and all paths originating at $v$.
  
  Thus, we have shown that $r_u$ is a pgf statistic. By writing $r$ as a sum of $r_u$ over $u\in G$, we see
  that $r$ is also a pgf statistic by \thref{lem:linearity}, and by the same argument as in the
  proof of \thref{lem:visited}, this statistic is continuous.
\end{proof}

Having established our comparison theorem and that several frog model statistics are icv and pgf statistics, we can now establish the relevant corollaries from Section \ref{sec:intro}.

\begin{proof}[Proofs of Corollaries~\ref{cor:det} and \ref{cor:trans}]
  We apply \thref{thm:comparison}, \thref{lem:returns}, and
\thref{prop:maximals}\ref{part:Jensen}, along with the observation made in Section~\ref{sec:orders} that  
$\P[X=\infty]\leq\P[Y=\infty]$ if $X\icvprec Y$.
\end{proof}

\begin{proof}[Proof of \thref{cor:treek}]
  The transience part of this result is a consequence of \cite[Proposition~15]{HJJ2}.
  The recurrence part follows from \thref{cor:det}, \cite[Theorem~1]{HJJ2}, and
  \cite[Theorem~1]{JJ3_log}.
\end{proof}

\begin{proof}[Proof of \thref{cor:Zd.transience}]
  This is proven the same way as Corollaries~\ref{cor:det} and \ref{cor:trans},
  except that \thref{prop:maximals}\ref{part:Bernoulli} is used instead of 
  \thref{prop:maximals}\ref{part:Jensen}.
\end{proof}

\begin{proof}[Proof of \thref{cor:Zd.shape}]
%We proceed by contradiction. Suppose that there exists a point $x_0 \in \mathbb Q^d \cap (\mathcal A' \setminus \mathcal A)$. Let $a_{nx_0} 	
For $v \in \mathbb Z^d$, let $T(v)$ and $T'(v)$ be the time that the vertex $v$ is activated for the frog models with i.i.d.-$\pi$ and i.i.d.-$\pi'$ frogs per site, respectively. Let $\viu{t}{v} = \mathbf 1\{T(v) \leq t\}$ be an indicator that $v$ has been activated by time $t$, and similarly for $\viu{t}{v}'$. 
By \thref{lem:visited},  $\viu{t}{v}$ and $\viu{t}{v}'$ are continuous pgf statistics. Moreover, we can express $T(v)$ and $T'(v)$ in terms of these statistics:
\begin{align*}
T(v) &= \textstyle \sum_{t=0}^\infty (1- \viu{t}{v}), \qquad  T'(v) = \sum_{t=0}^\infty (1- \viu{t}{v}').
\end{align*}
By \thref{thm:comparison}, we have $\viu{t}{v}\pgfprec\viu{t}{v}'$,
and hence $\E \viu{t}{v}\leq \E\viu{t}{v}'$. Apply this, along with Fubini's theorem, to the expressions for $T(v)$ and $T'(v)$ to obtain
\begin{align}\E T'(v) \leq \E T(v)\label{eq:ET}.\end{align}

In the proof of \cite[Theorem~1.1]{random_shape}, the limiting shapes are determined by functions $\mu$ and $\mu'$ with domain $\mathbb R^d$ defined via Kingman's subadditive ergodic theorem by
\begin{align}
\mu(v) = \lim_{n \to \infty} \f{T(nv)}{n}= \inf_{n \geq 1} \f{\E T(nv)}{n}, \label{eq:mu} \end{align}
 where $v \in \mathbb Z^d$, with $\mu'$ defined analogously. 
 After interpolating to all of $\mathbb R^d$ (see Step~5 in the proof of \cite[Theorem~1.1]{random_shape}), 
 the limiting shape is given by
$\mathbf{A} = \{x \in \mathbb R^d \colon \mu(x) \leq 1\}.$
The set $\mathbf{A}'$ is obtained in the same fashion. To deduce that $\mathbf{A} \subseteq \mathbf{A}'$, it then suffices to show that $\mu'(v) \leq \mu(v)$ for $v \in \mathbb Z^d$. This follows from \eqref{eq:ET} applied to the expected value formulation of $\mu$ in \eqref{eq:mu}.
\end{proof}

\begin{proof}[Proof of \thref{cor:death}]
  This also has the same proof as Corollaries~\ref{cor:det}, \ref{cor:treek}, and \ref{cor:trans}, 
  except that \thref{lem:visited} replaces \thref{lem:returns}.
\end{proof}
 
\begin{remark}\thlabel{rmk:pgf.applications}
  As we mentioned in the introduction, one of the motivations for part~\ref{item:comparison.pgf} of 
  \thref{thm:comparison} is that Lemma~10 from \cite{HJJ1} and similar results are direct corollaries of it. 
  Besides providing a satisfying explanation of why Lemma~10 holds, this is potentially useful
  in deriving other recurrence results. 
  
  Here, we describe Lemma~10 in more detail and explain why it
  follows from \thref{thm:comparison}\ref{item:comparison.pgf}.
  The lemma is a monotonicity result for an operator $\Aa$ acting on probability distributions on the
  nonnegative integers. (In \cite{HJJ1}, the operator is described as acting on 
  probability generating functions, but this comes to the same thing.) The operator can be defined as
  follows. Let $\pi$ be a probability distribution
  on the nonnegative integers. Consider a binary tree truncated to four vertices as in 
  Figure~\ref{fig:truncated.tree}. Place one frog on $\varnothing$ and one frog on $\varnothing'$,
  and independently sample from $\pi$ to decide the number of frogs on $u$ and $v$. The frog paths
  are random nonbacktracking walks stopped when a frog reaches a leaf. Now, run the frog model starting
  with the frog at $\varnothing$ active until all frogs are stopped. Define $\Aa\pi$ to be the distribution
  of frogs terminating at $\varnothing$. In \cite{HJJ1}, the operator was defined in a different way, 
  but it turns out to be equivalent. See also \cite[Section~2.2]{HJJ2} and \cite[Section~3.1.2]{JJ3_log}
  for similar constructions.
  
  The result of Lemma~10 is that $\pi\pgfprec\pi'$ imples that $\Aa\pi\pgfprec\Aa\pi'$.
  As $\Aa\pi$ is the distribution of visits to $\varnothing$ in the frog model on the truncated graph,
  it is a continuous pgf statistic by \thref{lem:returns}. Thus the lemma is a consequence of
  \thref{thm:comparison}\ref{item:comparison.pgf}.
\end{remark}
    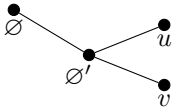
\begin{figure}
    \begin{center}
\begin{tikzpicture}[scale = 1,tv/.style={circle,fill,inner sep=0,
    minimum size=0.15cm,draw},walklines/.style={very thick, blue,bend left=15}]
    \path (0,0) node[tv] (R) {}
        (1,-0.6) node[tv] (L1) {}
        (1, 0.6)  node (L2) {}
        (2, -0.2) node[tv] (L12) {}
        (2,-1) node[tv] (L11) {}
        (L12)+(1,-0.3) coordinate (L121)
             +(1, 0.3) coordinate (L122)
        (L11)+(1,-0.3) coordinate (L111)
             +(1, 0.3) coordinate (L112);
  \draw 
       (R)--(L1);
\node[below] at (R){$\emptyset$};
\node[below] at (L12){$u$};
\node[below] at (L11){$v$};
\node[below] at (L1){$\emptyset'$\;\;\; };
    \draw (L1)--(L11);
       \draw (L1) -- (L12);

\end{tikzpicture}
    \end{center}
    \caption{A graph used to define the operator $\Aa$ in \cite{HJJ1}.}
    \label{fig:truncated.tree}

    \end{figure}

\section*{Appendix}

\begin{proof}[Proof of \thref{prop:discrete.ineq}]
  If $\varphi(x)$ is increasing and concave on $[0,\infty)$, then it is easily seen that
  $D\varphi(k)\geq 0$ and $D^2\varphi(k)\leq 0$ for all $k$. Similarly, one can
  easily check that if $\varphi(x)=1-p^x$, then $(-1)^nD^n\varphi(k)\leq 0$ for all
  $k$. This proves that the criteria stated in \ref{item:discrete.ineq.icv}
  and \ref{item:discrete.ineq.pgf} imply icv and pgf dominance, respectively.
  
  For the other direction in \ref{item:discrete.ineq.icv}, suppose that
  $X\icvprec Y$. Let $\varphi$ be a test function defined on the nonnegative integers
  satisfying $D \varphi(k)\geq 0$ and $D^2\varphi(k)\leq 0$ for all $k$. This can be
  extended to an increasing concave function on $[0,\infty)$ by linearly interpolating
  between integer points, for example, and hence $\E\varphi(X)\leq\E\varphi(Y)$
  by the assumption that $X\icvprec Y$.
  
  For \ref{item:discrete.ineq.pgf}, suppose that $X\pgfprec Y$, and let $\varphi$ be a bounded test
  function on the nonnegative integers satisfying $(-1)^n D^n\varphi(k)\leq 0$ for
  all $k$. We now appeal to a classic
  result of Hausdorff's stating that a sequence $f(0), f(1),\ldots$ can be represented as
  a moment sequence
  \begin{align*}
    f(k) &= \int_{[0,1]}u^k\,\sigma(du)
  \end{align*}
  for some positive measure $\sigma$ if and only if
  \begin{align*}
    (-1)^n D^n f(k)\geq 0
  \end{align*}
  for all $n\geq 0$. (See \cite[Theorem~2.6.4]{Akh}, and note that $\Delta^n$ defined there
  is equal to $(-1)^nD^n$.)
  Let $C$ be an upper bound on $\varphi$, and apply Hausdorff's result to
  $C-\varphi(k)$ to obtain the representation
  \begin{align*}
    C - \varphi(k) &= \int_{[0,1]}u^k\,\sigma(du)
  \end{align*}
  for some measure $\sigma$.
  Defining
  \begin{align*}
    \psi(x) &= \int_{[0,1]}u^x\,\sigma(du),
  \end{align*}
  for $x>0$ and extending the function continuously to $x=0$, we obtain a completely monotone function
  $\psi$ satisfying
  \begin{align*}
    \psi(k) &= C-\varphi(k),\qquad k\in\{1,2,\ldots\},\\
    \psi(0) &= \int_{(0,1]}\sigma(du)\leq C-\varphi(0).
  \end{align*}
  Now, we evaluate
  \begin{align*}
    \E \varphi(X) &= C - \E\psi(X) - \P[X=0]\bigl(C - \psi(0)-\varphi(0)\bigr),\\
    \E \varphi(Y) &= C - \E\psi(Y) - \P[Y=0]\bigl(C - \psi(0)-\varphi(0)\bigr).
  \end{align*}
  Since $\psi$ is completely monotone and $X\pgfprec Y$, we have $\E \psi(X)\geq\E\psi(Y)$.
  The relation $X\pgfprec Y$ implies that $\P[X=0]\geq\P[Y=0]$, and together with
  $C - \psi(0)-\varphi(0)\geq 0$, this implies that $\E\varphi(X)\leq\E\varphi(Y)$.
\end{proof}

\begin{proof}[Proof of \thref{prop:maximals}]
  Part~\ref{part:Jensen} follows immediately from Jensen's inequality.
  For part~\ref{part:Bernoulli}, 
  let $\varphi$ be an arbitrary increasing concave function on $[0,\infty)$. 
  To simplify the algebra, let
  $U=X-k$, $V=Y-k$, and $\psi(x)=\varphi(x+k)-\varphi(k)$.
  With these replacements, our goal is to show that
  $\E \psi(U)\leq \E\psi(V)$. We know that $\E U\in[0,1]$
  and that $V$ is Bernoulli, and we know that $\psi$ is increasing and concave on
  $[-k,\infty)$ and satisfies $\psi(0)=0$.

  Since $V$ is Bernoulli with mean at least $\E U$, 
  \begin{align}\label{eq:bern}
    \E\psi(V) \geq (\E U)\psi(1).
  \end{align}
  
  Define
  \begin{align*}
    a &= \E[U\mid U\leq 0], & p&=\P[U\leq 0],\\
    b &= \E[U\mid U\geq 1], & q=1-p&=\P[U\geq 1].
  \end{align*}
  If $p=0$ or $q=0$, then $U$ is deterministic and the result is trivial because $U$ and $V$
  have the same distribution. Thus we can assume that
  both conditional expectations above are well defined.
  
  Applying Jensen's inequality,
  \begin{align}\label{eq:jns}
    \E \psi(U) &= p\E[\psi(U)\mid U\leq 0] + q\E[\psi(U)\mid U\geq 1]
                  \leq p\psi(a) + q\psi(b).
  \end{align}
  As $a\leq 0$ and $b\geq 1$, the points $(a,\psi(a))$ and $(b,\psi(b))$ lie under the secant line
  connecting $(0,0)$ and $(1,\psi(1))$ by the concavity of $\psi$. Thus $\psi(a)\leq a\psi(1)$
  and $\psi(b)\leq b\psi(1)$. Applying this to \eqref{eq:jns} and combining with \eqref{eq:bern}
  gives
  \begin{align*}
    \E\psi(U) &\leq (pa + qb)\psi(1) = (\E U)\psi(1)\leq \E\psi(V).\qedhere
  \end{align*}
\end{proof}

To prove \thref{lem:composition}, we will need a few technical statements that we will use to replace
our discrete derivatives with continuous ones.
For a function $g$ on $\RR^n$, let $\partial_ig$ denote the partial derivative
with respect to the $i$th coordinate, and let $\Delta_i g$ denote the discrete derivative in
the $i$th coordinate; that is,
\begin{align*}
  \Delta_i g(x_1,\ldots,x_n) &= g(x_1,\ldots,x_i+1,\ldots,x_n) - g(x_1,\ldots,x_n).
\end{align*}    
%For $B=\{i_1,\ldots,i_k\}\subseteq[n]$ with $i_1,\ldots,i_k$ distinct, let 
%$\partial_B=\partial_{i_1}\cdots\partial_{i_k}$ and $\Delta_B=\Delta_{i_1}\cdots\Delta_{i_k}$.

\begin{lemma}\thlabel{lem:multilinearize}
  For any $g\colon\{0,1\}^n\to\RR$, there is a unique multilinear polynomial $p(x_1,\ldots,x_n)$
  that matches $g$ on $\{0,1\}^n$. Furthermore, for distinct $b_1,\ldots,b_k\in[n]$
  and for $x_1,\ldots,x_n\in\{0,1\}$,
  \begin{align}\label{eq:multilinearize.derivatives}
    \partial_{b_1}\cdots\partial_{b_k} p(x_1,\ldots,x_n)&=\Delta_{b_1}\cdots\Delta_{b_k}
    g(x_1\1\{1\notin B\},\ldots,x_n\1\{n\notin B\}).
  \end{align}
\end{lemma}
\begin{proof}
  We construct $p$ as
  \begin{align*}
    p(x_1,\ldots,x_n) = \sum_{(t_1,\ldots,t_n)\in\{0,1\}^n} g(t_1,\ldots,t_n)\prod_{i=1}^n 
                           \Bigl( (1-x_i)\1_{t_i=0} + x_i\1_{t_i=1}\Bigr),
  \end{align*}
  which is multilinear and
  matches $g$ when evaluated on $\{0,1\}^n$. 
  Multilinear polynomials on $n$ variables form a $2^n$-dimensional vector space.
  The map $g\mapsto p$ given above is then an injective
  linear map between $2^n$-dimensional vector spaces. Hence, it is a bijection,
  showing uniqueness of $p$.

  For any multilinear polynomial $p$, observe that $\partial_i p(x_1,\ldots,x_n)$ does not depend on $x_i$. 
    Thus
    \begin{align*}
      \partial_i p(x_1,\ldots,x_n)
        &=\int_0^1 \partial_ip(x_1,\ldots,x_{i-1},h,x_{i+1},\ldots,x_n)\,dh\\
        &= \Delta_i p(x_1,\ldots,x_{i-1},0,x_{i+1},\ldots,x_n).
    \end{align*}
    Observing that $\partial_i$ and $\Delta_j$ commute, repeated application of this shows that
    \begin{align*}
      \partial_{b_1}\cdots\partial_{b_k} p(x_1,\ldots,x_n)&=\Delta_{b_1}\cdots\Delta_{b_k}
          p(x_1\1\{1\notin B\},\ldots,x_n\1\{n\notin B\}),
    \end{align*}
    which proves \eqref{eq:multilinearize.derivatives}.
\end{proof}

\begin{lemma}\thlabel{lem:multilinear.max}
  If $p$ is a multilinear polynomial, then its maximum and minimum on $[0,1]^n$
  are attained on $\{0,1\}^n$.
\end{lemma}
\begin{proof}
  As we mentioned,
  $\partial_ip(x_1,\ldots,x_n)$ does not depend on $x_i$. The function $p(x_1,\ldots,x_n)$ is 
  therefore monotone in $x_i$ with the other coordinates held fixed. Thus, if $p$ achieves its
  maximum at $(x_1,\ldots,x_n)$, it must also achieve it either at $(0,x_2,\ldots,x_n)$
  or $(1,x_2,\ldots,x_n)$. Repeating the argument it must also achieve it
  with $x_2$ set to $0$ or $1$, and so on. The identical argument applies to the minimum.
\end{proof}

\begin{proof}[Proof of \thref{lem:composition}]
  Fix $(\eta,S)$ and frog paths $P^1_\proc,\ldots,P^n_\proc$ starting at the same vertex.
  We need to show that
  \begin{align}\label{eq:composition.goal}
    (-1)^n\Delta_{P^1_\proc}\cdots\Delta_{P^n_\proc}(\varphi\circ f)(\eta,S)\leq 0,
  \end{align}
  for $n=1,2$ in case~\ref{item:composition.icv} and for all $n\geq 1$ in 
  case~\ref{item:composition.pgf}. Note that conditions~\ref{infinite.increasing}
  and~\ref{infinite.concave} of the definition of icv and pgf statistics are irrelevant here, since
  $\varphi$ is bounded and $\varphi(\infty)$ is interpreted as $\lim_{x\to\infty}\varphi(x)$.
  
  First, we dispense with the cases where $f$ takes an infinite value.
  If $f(\eta,S)=\infty$, then $f$ continues to take the value $\infty$ when additional frogs
  are added at the same vertex by condition~\ref{infinite.increasing} of the definition.
  It is easy to confirm that the left hand side of \eqref{eq:composition.goal} is then zero.
  For the case that $f(\eta,S)<\infty$ but $f$ takes the value $\infty$ in some term in the left hand side
  of \eqref{eq:composition.goal}, suppose that $P^1_\proc,\ldots,P^n_\proc$ is a minimal set of paths
  so that \eqref{eq:composition.goal} fails.
  By condition~\ref{infinite.concave}, 
  there exists $P^i_\proc$ with $f\bigl(\sigma_{P^i_\proc}(\eta,S)\bigr)=\infty$.
  Take $i=n$ for simplicity.
  Expanding $\Delta_{P^n_\proc}$,
  \begin{align*}
    \Delta_{P^1_\proc}\cdots\Delta_{P^n_\proc}(\varphi\circ f)(\eta,S)
      &= \Delta_{P^1_\proc}\cdots\Delta_{P^{n-1}_\proc}(\varphi\circ f)\bigl(\sigma_{P^n_\proc}(\eta,S)\bigr)
        - \Delta_{P^1_\proc}\cdots\Delta_{P^{n-1}_\proc}(\varphi\circ f)(\eta,S).
  \end{align*}
  The first term on the right hand side is zero, by the argument for the case where $f(\eta,S)=\infty$.
  By minimality of $\{P^1_\proc,\ldots,P^n_\proc\}$, the second term satisfies \eqref{eq:composition.goal}
  with $n$ replaced by $n-1$, which confirms \eqref{eq:composition.goal}.

  For the rest of the proof, we assume that $f$ takes only finite values.
  For $x_1,\ldots,x_n\in \{0,1\}$, let $g(x_1,\ldots,x_n)$ be given by evaluating $f$ on the frog model
  $(\eta, S)$ with frogs $P^i_\proc$ added for each $x_i=1$. In the more formal notation of
  \eqref{eq:sigma.U}, we let   
  $U_x=\{i\in[n]\colon x_i=1\}$ and define
  \begin{align*} 
    g(x_1,\ldots,x_n) &= f(\sigma_{U_x}(\eta,S)).
  \end{align*}
  Let $p$ be the multilinear polynomial matching $g$ on $\{0,1\}^n$.
  For any $x_1,\ldots,x_n\in\{0,1\}$ and $B=\{b_1,\ldots,b_k\}\subseteq[n]$,
  let $V=\{i\in[n]\colon x_i=1\}\setminus B$.
  By \thref{lem:multilinearize},
  \begin{align*}
    \partial_{b_1}\cdots\partial_{b_k} p(x_1,\ldots,x_n)&=\Delta_{b_1}\cdots\Delta_{b_k}
    g\bigl(x_1\1\{1\notin B\},\ldots,x_n\1\{n\notin B\}\bigr)\\
    &= \Delta_{P^{b_1}_\proc}\cdots \Delta_{P^{b_k}_\proc} f\bigl(\sigma_V(\eta, S)\bigr).
  \end{align*}
  Thus, by our assumption on $f$,
  \begin{align}\label{eq:multi.pgf}
    (-1)^n\partial_{b_1}\cdots\partial_{b_k} p(x_1,\ldots,x_n)\leq 0
  \end{align}
  for $k=1,2$ in case~\ref{item:composition.icv} and for $k\geq 1$ in case~\ref{item:composition.pgf}.
  This holds for all $(x_1,\ldots,x_n)\in\{0,1\}^n$ and all $B\subseteq [n]$. Since partial derivatives
  of multilinear polynomials are also multilinear polynomials, \thref{lem:multilinear.max}
  applies and shows that \eqref{eq:multi.pgf} holds for all $(x_1,\ldots,x_n)\in[0,1]^n$.
  
  Now, we finish the job of transferring the problem to work exclusively with $p$
  rather than $f$.  
  By definition of $p$ in the first line and an easy induction in the second,
  \begin{align*}
    \Delta_{P^1_\proc}\cdots\Delta_{P^n_\proc}(\varphi\circ f)(\eta,S)
      &=\Delta_1\cdots\Delta_n(\varphi\circ p)(0,\ldots,0)\\
      &=\int_0^1\cdots\int_0^1\partial_1\cdots\partial_n(\varphi\circ p)(x_1,\ldots,x_n)\,dx_1\cdots dx_n.
  \end{align*}
  Thus, to prove \eqref{eq:composition.goal} it suffices to show that
  \begin{align}\label{eq:composition.final.goal}
    (-1)^n\partial_1\cdots\partial_n(\varphi\circ p)(x_1,\ldots,x_n)\leq 0
  \end{align}
  for $n=1,2$ in case~\ref{item:composition.icv} and for $n\geq 1$ in
  case~\ref{item:composition.pgf}.
  This follows from the multivariate Fa\`a di Bruno formula, the chain rule
  for higher derivatives (see \cite[Proposition~1]{Hardy} for a reference). In full detail,
  the formula states that
  \begin{align}\label{eq:FdB}
    \partial_1\cdots\partial_n(\varphi\circ p)(x_1,\ldots,x_n) &=
      \sum_{\pi} \varphi^{(\abs{\pi})}\bigl(p(x_1,\ldots,x_n)\bigr)
           \prod_{B\in\pi}\partial_{b_1}\cdots\partial_{b_k} p(x_1,\ldots,x_n),
  \end{align}
  where the sum is over
  all set partitions $\pi$ of $[n]$ and $b_1,\ldots,b_k$ are the elements of $B$.
  Now, we fix some partition $\pi$ and determine the sign of its term in the sum.
  In case~\ref{item:composition.pgf},
  by our assumption that $(-1)^m\varphi^{(m)}\leq 0$ and by \eqref{eq:multi.pgf}, its sign is
  \begin{align*}
    (-1)^{\abs{\pi}+1} \prod_{B\in\pi}(-1)^{\abs{B}+1}
      &= (-1)^{\abs{\pi}+1}(-1)^{n}(-1)^{\abs{\pi}}=(-1)^{n+1},
  \end{align*}
  confirming \eqref{eq:composition.final.goal}.
  In case~\ref{item:composition.icv}, this proof also applies, since each partition $\pi$ is into
  at most two blocks and each block has at most two elements.
\end{proof}

\bibliographystyle{amsplain}
\bibliography{frog_paper_iii}

\end{document}